\documentclass[12pt,reqno]{amsart}
\usepackage{amsmath,amsthm,amssymb,amsfonts,amscd}
\usepackage{mathrsfs}
\usepackage{bbm}
\usepackage{bbding}
\usepackage{hyperref}
\usepackage{geometry}
\usepackage{color}
\usepackage{xcolor}

\usepackage{picture,epic}
\usepackage{tikz}

\numberwithin{equation}{section}

\theoremstyle{plain}
\newtheorem{theorem}{Theorem}[section]
\newtheorem{lemma}[theorem]{Lemma}
\newtheorem{corollary}[theorem]{Corollary}
\newtheorem{proposition}[theorem]{Proposition}

\theoremstyle{definition}
\newtheorem{conjecture}[theorem]{Conjecture}

\theoremstyle{remark}
\newtheorem{remark}[theorem]{Remark}

\renewcommand{\Re}{\operatorname{Re}}
\renewcommand{\Im}{\operatorname{Im}}
\newcommand{\vol}{\operatorname{vol}}

\newcommand{\sym}{\operatorname{sym}}

\newcommand{\GL}{\operatorname{GL}}
\newcommand{\SL}{\operatorname{SL}}

\renewcommand{\mod}{\operatorname{mod}\ }

\newcommand{\dd}{\mathrm{d}}

\makeatletter
\def\@tocline#1#2#3#4#5#6#7{\relax
  \ifnum #1>\c@tocdepth 
  \else
    \par \addpenalty\@secpenalty\addvspace{#2}%
    \begingroup \hyphenpenalty\@M
    \@ifempty{#4}{%
      \@tempdima\csname r@tocindent\number#1\endcsname\relax
    }{%
      \@tempdima#4\relax
    }%
    \parindent\z@ \leftskip#3\relax \advance\leftskip\@tempdima\relax
    \rightskip\@pnumwidth plus4em \parfillskip-\@pnumwidth
    #5\leavevmode\hskip-\@tempdima
      \ifcase #1
       \or\or \hskip 1em \or \hskip 2em \else \hskip 3em \fi%
      #6\nobreak\relax
    \hfill\hbox to\@pnumwidth{\@tocpagenum{#7}}\par
    \nobreak
    \endgroup
  \fi}
\makeatother

\begin{document}

\title
{Joint distribution of Hecke eigenforms}

\author{Bingrong Huang}


\address{Data Science Institute \& State Key Laboratory of Cryptography and Digital Economy Security \\ Shandong University \\ Jinan \\ Shandong 250100 \\China} 

\email{brhuang@sdu.edu.cn}


\date{\today}

\begin{abstract}
  In this paper, we formulate conjectures on the joint distribution of several Hecke eigenforms.
  We prove an asymptotic formula of the joint mass of two Hecke eigenforms
  under the generalized Riemann Hypothesis (GRH) and the generalized Ramanujan conjecture (GRC).
  We also show that a higher decorrelation of two Hecke eigenforms asymptotically vanishes under GRH.
  As a consequence, we prove an asymptotic formula for the first moment of the triple product $L$-functions under GRH and GRC.
\end{abstract}



\thanks{This work was supported by  the National Key R\&D Program of China (No. 2021YFA1000700) and 
the Scientific Research Innovation Capability Support Project for Young Faculty (No. SRICSPYF-ZY2025158).}


\maketitle

\section{Introduction} \label{sec:Intr}

The limiting distribution of values of eigenfunctions is a central problem in analytic number theory and quantum chaos. Let $\mathbb{H}=\{x+iy:x\in \mathbb{R},\ y>0\}$ be the upper half plane and $\Gamma=\SL_2(\mathbb{Z})$ the modular group. Let $k\geq12$ be an even integer.
Let $H_k$ denote a Hecke basis of the space $S_k$ of all holomorphic cusp forms of weight $k$ on the modular surface $\Gamma\backslash \mathbb{H}$.
The Petersson inner product of $f,g\in S_k$ is defined by
\[
  \langle f,g\rangle := \int_{\Gamma\backslash \mathbb{H}} y^k f(z) \overline{g(z)} \dd\mu z,
\]
where 
$\dd\mu z= \dd x\dd y/y^2$ is the hyperbolic measure.
In the large weight limit, we have the following locally Gaussian value distribution conjecture.

\begin{conjecture}\label{conj:rwc}
  Let $\Omega$ be a fixed compact set of $\Gamma\backslash\mathbb{H}$,
  such that the boundary $\partial\Omega$ has hyperbolic measure zero.
  Let $f\in H_k$ be a Hecke eigenform  such that
  \[
    \frac{1}{\vol(\Gamma\backslash \mathbb{H})}\int_{\Gamma\backslash \mathbb{H}}  y^k |f(z)|^2 \dd\mu z = 1.
  \]
  Then for any $a\in \mathbb{N}$, we have
  \[
    \frac{1}{\vol(\Omega)} \int_{\Omega} y^{ak} |f(z)|^{2a}  \dd \mu z = a! + o(1),
  \]
  as $k\rightarrow \infty$.
\end{conjecture}

\begin{remark}
  Let $Z\sim \mathcal{CN}(0,1)$ be a standard complex normal distribution. Then for any $a\in \mathbb{N}$, we have
  $\mathop{\mathbb{E}}[|Z|^{2a}]=a!$.
\end{remark}

\begin{remark}
  The compactness assumption of $\Omega$ is necessary. Xia \cite{Xia} proved that
  \[
    \max_{z\in\mathbb{H}} |y^{k/2}f(z)| = k^{1/4+o(1)}.
  \]
  Moreover, Blomer, Khan, and Young \cite{BlomerKhanYoung2013distribution} showed that for $p>6$,
  \[  \int_{\Gamma\backslash \mathbb{H}} |y^{k/2}f(z)|^p \dd\mu z   \gg k^{p/4-3/2-\varepsilon}. \]
\end{remark}

For $a=1$, this is the holomorphic analog of the quantum unique ergodicity (hQUE) conjecture of Rudnick and Sarnak \cite{RS}. hQUE was solved by Holowinsky and Soundararajan \cite{Holowinsky,Soundararajan,HS}, which confirmed the equidistribution of mass of $f$ in the $L^2$ sense.

For $a=2$, this is related to the  $L^4$-norm problem of Hecke eigenforms. When $\Omega=\Gamma\backslash\mathbb{H}$, this is
\cite[Conjecture 1.2]{BlomerKhanYoung2013distribution}.
Blomer, Khan, and Young \cite{BlomerKhanYoung2013distribution}  proved that
\[
  \int_{\Gamma\backslash \mathbb{H}} y^{2k}|f(z)|^{4}  \dd \mu z
  = O( k^{1/3+\varepsilon}),
\]
for any positive $\varepsilon$.
Assuming the generalized Riemann Hypothesis (GRH), Zenz \cite{Zenz} proved
\[
  \int_{\Gamma\backslash \mathbb{H}} y^{2k}|f(z)|^{4}  \dd \mu z = O(1).
\]
It seems hard to deduce Conjecture \ref{conj:rwc} for $a=2$,
even if we assume GRH and take $\Omega=\Gamma\backslash\mathbb{H}$.
Note that for the Maass form case, Buttcane--Khan \cite{ButtcaneKhan} proved an asymptotic formula for the $L^4$-norm under the generalized Lindel\"of Hypothesis (GLH).

\medskip

In this paper, our main interest is the joint  distribution of Hecke eigenforms in the large weight limit.
Let $f\in H_k$ and $g\in H_\ell$ and assume $\langle f,g\rangle=0$ if $k=\ell$.
We will consider the joint mass distribution of $|y^{k/2}f(z)|$ and $|y^{\ell/2}g(z)|$, which extends Conjecture \ref{conj:rwc}.
We  will also consider the joint distribution of $y^{k/2}f(z)$ and $y^{\ell/2}g(z)$ when $k=\ell$, trying to catch the sign changes of both Hecke eigenforms. We view the later one as higher power extensions of the decorrelation of Hecke eigenforms \cite{Constantinescu,Huang}.
One may extend our ideas to consider joint  distribution of several Hecke eigenforms.
As a consequence of our result on the joint mass distribution, we prove an asymptotic formula for a first moment of the triple product $L$-function under GRH and GRC, from which we obtain a nonvanishing result for these $L$-functions.

Kurlberg--Rudnick \cite{KurlbergRudnick} studied value distribution for eigenfunctions of the quantum cat map.
Recently, Hua--Huang--Li \cite{HHL} proved a similar result on the mixed $(2,2)$ moment of two Hecke--Maass cusp forms. 
This paper can be viewed as a continuation of \cite{HHL}. Since the norm problems of holomorphic forms are usually different from Maass forms (e.g. QUE and $L^4$-norm), and we have some other results for the holomorphic setting, we decide to write this paper separately. 

\subsection{Joint mass distribution of Hecke eigenforms}

We first consider the joint mass distribution of several Hecke eigenforms. We formulate the following conjecture, which indicates that orthogonal Hecke eigenforms are statistically independent.

\begin{conjecture}\label{conj:jointmass}
  Let $\Omega$ be a fixed compact set of $\Gamma\backslash\mathbb{H}$,
  such that the boundary $\partial\Omega$ has hyperbolic measure zero.
  Let $J\in \mathbb{Z}_{>1}$.
  Let $f_j\in H_{k_j}$  with $\langle f_j,f_j\rangle = \vol(\Gamma\backslash\mathbb{H})$ for all $1\leq j\leq J$. Assume that $\langle f_i,f_j\rangle=0$ if $k_i=k_j$ and $i\neq j$.  Let $a=(a_j) \in \mathbb{Z}_{\geq1}^J$, then we have
  \[
    \frac{1}{\vol(\Omega)}  \int_{\Omega} y^{\sum_{j=1}^J a_j k_j} \prod_{j=1}^J |f_j(z)|^{2a_j}  \dd \mu z  = \prod_{j=1}^{J} (a_j)!  + o(1)
  \]
  as $\min k_j\rightarrow \infty$.
\end{conjecture}

For $J=2$, $(a_1,a_2)=(1,1)$ and $\Omega=\Gamma\backslash \mathbb{H}$, we can prove the above conjecture by assuming the generalized Riemann Hypothesis (GRH) and the generalized Ramanujan conjecture (GRC). Here we will use GRC for $\GL(2)$ Hecke--Maass cusp forms, which asserts that the $n$-th Hecke eigenvalue of a $\GL(2)$ Hecke--Maass cusp form is bounded by the number of divisors of $n$. Our first main result is the following asymptotic formula.

\begin{theorem}\label{thm:221}
  Assume GRH and GRC. Let $f\in H_k$ and $g\in H_\ell$ with $\langle f,f\rangle = \langle g,g\rangle = \vol(\Gamma\backslash\mathbb{H})$. Assume 
  $\langle f,g\rangle=0$ if $k=\ell$.   Then  we have
  \begin{equation}\label{eqn:|f|^2|g|^2}
    \frac{1}{\vol(\Gamma\backslash \mathbb{H})}  \int_{\Gamma\backslash \mathbb{H}} y^{k+\ell} |f(z)|^{2}  |g(z)|^{2} \dd \mu z   = 1 + O((\log (k+\ell))^{-1/4+\varepsilon })
  \end{equation}
  as $\max (k,\ell) \rightarrow \infty$.
  In particular, Conjecture \ref{conj:jointmass} is true when $J=2$, $(a_1,a_2)=(1,1)$ and $\Omega=\Gamma\backslash \mathbb{H}$, under GRH and GRC.
\end{theorem}

Note that we only require $\max (k,\ell) \rightarrow \infty$, which is weaker than the condition in Conjecture \ref{conj:jointmass}.
This may be special for the exponent $2$ of forms and $\Omega=\Gamma\backslash \mathbb{H}$, since we assume the $L^2$-normalization of cusp forms. Also, when $\min(k,\ell)\ll (\max(k,\ell))^{1-\varepsilon}$, we can actually prove power saving error term under GLH (see \S\ref{sec:joint-mass}).

Theorem \ref{thm:221} may be compared with \cite[Theorem 1.1]{Zenz}, where Zenz  considered the case $f=g$.
In this case, one can only prove an upper bound by assuming GRH.
Note that $fg\in S_{k+\ell}$.
It is natural to view the left hand side of \eqref{eqn:|f|^2|g|^2} as $\langle fg,fg \rangle$. By decomposing $fg$ into a Hecke basis of holomorphic forms and Parseval's identity, we have
\begin{equation*}
  \langle fg , fg \rangle
  =\sum_{h\in H_{k+\ell}} |\langle fg , h\rangle|^2.
\end{equation*}
An application of conditional bounds for moments of $L$-functions will show that  $\langle fg , fg \rangle=O(1)$ under GRH.

Our novelty to prove an asymptotic formula for $\langle fg , fg \rangle$ (Theorem \ref{thm:221})
is that we instead use spectral decompositions for $y^k|f|^2$ and $y^\ell |g|^2$ into Hecke--Maass forms.
Such an approach to study the $L^4$-norm problem was mentioned in Blomer--Khan--Young \cite[P. 2615]{BlomerKhanYoung2013distribution}, where as they wrote is ``more difficult''. 
It turns out that in our case this may be a better way. 
We  notice that  $\langle fg , fg \rangle=\langle |F|^2, |G|^2 \rangle$ where $F(z)=y^{k/2}f(z)$ and $G(z)=y^{\ell/2}g(z)$, and then use the decomposition of $L^2(\Gamma\backslash \mathbb{H})$ by the constant function, the Hecke--Maass cusp forms of weight zero, and the Eisenstein series.
Indeed, let $u_0=\vol(\Gamma\backslash\mathbb{H})^{-1/2}$, $\{u_j\}_{j\geq1}$ be an $L^2$-normalized Hecke basis of the space of cuspidal Maass forms, and $E_t(z)=E(z,1/2+it)$ the Eisenstein series, then we have
\[
  \langle |F|^2, |G|^2 \rangle = \langle |F|^2, u_0 \rangle \langle u_0, |G|^2 \rangle + \sum_{j\geq1} \langle |F|^2, u_j \rangle \langle u_j, |G|^2 \rangle
  + \frac{1}{4\pi} \int_{\mathbb{R}} \langle |F|^2, E_t \rangle \langle E_t, |G|^2 \rangle \dd t.
\]
In this way, we see the main term coming from the constant function contribution. To bound the contribution from the Hecke--Maass cusp forms, we will use moments of $L$-functions. Indeed, by Watson's formula and Stirling's formula, under GRH, we have essentially
\begin{multline*}
  \sum_{j\geq1} \langle |F|^2, u_j \rangle \langle u_j, |G|^2 \rangle \\
  \ll  \frac{(\log (k+\ell))^\varepsilon}{\sqrt{k\ell}} \sum_{j\geq1}  \frac{L(1/2,f\times f\times u_j)^{1/2} L(1/2,g\times g\times u_j)^{1/2}}{L(1,\sym^2 u_j)}  e^{-t_j^2/(2k)-t_j^2/(2\ell)} .
\end{multline*}
Using the factorization
\[
  L(1/2,f\times f\times u_j) = L(1/2, u_j)L(1/2,\sym^2 f\times u_j)
\]
and nonnegativity of central $L$-values, we need to consider the following fractional moment of $L$-functions
\[
   \frac{(\log (k+\ell))^\varepsilon}{\sqrt{k\ell}} \sum_{j\geq1}  \frac{L(1/2,u_j) L(1/2,\sym^2 f\times u_j)^{1/2} L(1/2,\sym^2 g\times u_j)^{1/2}}{L(1,\sym^2 u_j)}   e^{-t_j^2/(2k)-t_j^2/(2\ell)} .
\]
We will apply Soundararajan's method \cite{soundararajan2009moments} on moments of $L$-functions under GRH and GRC to deal with the above fractional moments.
Since the proof uses Maass cusp forms, we will need GRC for them.
See also \cite{BB, BBK, HuangLester, LR}, and the references therein, for other applications of Soundararajan's method.

\subsection{Higher decorrelation of Hecke eigenforms}

Now we consider the joint  distribution of powers of Hecke eigenforms.
For simplicity, we only consider two Hecke eigenforms with the same weight.
We have the following conjecture.

\begin{conjecture}\label{conj:0}
  Let $\Omega$ be a fixed compact set of $\Gamma\backslash\mathbb{H}$,
  such that the boundary $\partial\Omega$ has hyperbolic measure zero.
  Let $f,g\in H_k$  such that  $\langle f,g\rangle=0$. For any positive integer $a\in \mathbb{N}$, we have
  \begin{equation}\label{eqn:fg^a}
    \int_{\Omega} y^{ak} f(z)^a  \overline{g(z)}^a \dd \mu z  = o(1)
  \end{equation}
  as $k\rightarrow \infty$.
\end{conjecture}

\begin{remark}
  Let $Z\sim \mathcal{CN}(0,1)$ be a standard complex normal distribution. Then for any $a\in \mathbb{N}$, we have
  $\mathop{\mathbb{E}}[Z^{a}]=0$. Hence we have $\mathop{\mathbb{E}}[Z^{a} \bar{W}^a]=0$, where $Z$ and $W$ are two independent complex Gaussians.  So we expect that there is no main term in \eqref{eqn:fg^a}.
\end{remark}

For $a=1$, this was proved by Constantinescu \cite{Constantinescu} and  Huang \cite{Huang}.

Assuming GRH, we will confirm the case $a=2$ when $\Omega=\Gamma\backslash \mathbb{H}$.
Our main result in this setting  is the following estimate.

\begin{theorem}\label{thm:220}
  Assume GRH. Let $f,g \in H_k$ such that  $\langle f,g\rangle=0$. Then  we have
  \[
    \int_{\Gamma\backslash \mathbb{H}} y^{2k} f(z)^2 \overline{g(z)}^2 \dd \mu z  = O((\log k)^{-1/4+\varepsilon})
  \]
  as $k\rightarrow \infty$. In particular, Conjecture \ref{conj:0} is true when $a=2$ and $\Omega=\Gamma\backslash \mathbb{H}$, under GRH.
\end{theorem}


Note that $f^2 ,g^2 \in H_{2k}$. To prove Theorem \ref{thm:220}, by Parseval's identity, we have
\[
  \int_{\Gamma\backslash \mathbb{H}} y^{2k} f(z)^2 \overline{g(z)}^2  \dd \mu z
  = \sum_{h\in H_{2k}} \langle f^2,h\rangle \langle h,g^2 \rangle.
\]
We should take absolute values on the right hand side to get an upper bound. By Watson's formula and Stirling's formula, we have
\[
  \int_{\Gamma\backslash \mathbb{H}} y^{2k} f(z)^2 \overline{g(z)}^2  \dd \mu z
  \ll  \frac{1}{k}\sum_{h\in H_{2k}} \frac{L(1/2,f\times f\times h)^{1/2}L(1/2,g\times g\times h)^{1/2}}{L(1,\sym^2 f)L(1,\sym^2 g)L(1,\sym^2 h)}.
\]
Using the factorization
\[
  L(1/2,f\times f\times h) = L(1/2, h)L(1/2,\sym^2 f\times h)
\]
and nonnegativity of central $L$-values, we need to consider the following fractional moment of $L$-functions
\[
  \frac{1}{k}\sum_{h\in H_{2k}} \frac{L(1/2,h)^{1/2}L(1/2,\sym^2 f\times h)^{1/2}L(1/2,\sym^2 g\times h)^{1/2}}{L(1,\sym^2 h)}.
\]
By Soundararajan's method again, this is bounded by $O((\log k)^{-1/4+\varepsilon})$ under GRH.

\subsection{Applications to the triple product $L$-functions}

Let  $\lambda_f(n)$ be the $n$-th Hecke eigenvalue of $f$, and let $\alpha_{f}(p)$ and $\alpha_{f}(p)^{-1}$ be the Satake parameters  of $f$ at prime $p$. We have $\lambda_f(p)=\alpha_f(p)+\alpha_f(p)^{-1}$. Define the symmetric square $L$-function of $f$ by
\[
  L(s,\sym^2 f) = \zeta(2s) \sum_{n\geq1} \frac{\lambda_f(n^2)}{n^s}, \quad \Re(s)>1.
\]
Let $f\in H_k$, $g\in H_\ell$ and $h\in H_{k+\ell}$. The triple product $L$-function of $f,g,h$ is defined  by
\[
  L(s,f\times g\times h) = \prod_{p} \prod_{a=\pm1}\prod_{b=\pm1}\prod_{c=\pm1} \left(1-\frac{\alpha_{f}(p)^{a}\alpha_{g}(p)^{b}\alpha_{h}(p)^{c}}{p^s}\right)^{-1}, \quad \Re(s)>1.
\]
All these $L$-functions have analytic continuation to the whole complex plane.
By Watson's formula for $|\langle fg , h\rangle|^2$, we know $\langle fg , fg\rangle$ is related to  the following first moment of the triple product $L$-values,
\[
  \frac{2\pi^2}{k+\ell-1} \sum_{h\in H_{k+\ell}} \frac{
    L(1/2,f\times g\times h)}{L(1,\sym^2 h)}.
\]
Unconditionally, Blomer, Khan, and Young \cite[Corollary 1.5]{BlomerKhanYoung2013distribution} proved the following nontrivial upper bound
\[
  O((k\ell)^{1/6+\varepsilon}).
\]
Assume GRH. By a similar argument as in Zenz \cite{Zenz}, one may prove a sharp upper bound
\[
  O(L(1,\sym^2 f) L(1,\sym^2 g)).
\]

However, as a consequence  of Theorem \ref{thm:221}, we are able to prove an asymptotic formula for the first moment when $f\neq g$, under GRH and GRC. We have the following theorem.

\begin{theorem}\label{thm:moment}
  Assume GRH and GRC. Let $k, \ell$ be two even positive integers, and let  $f\in H_k$ and $g\in H_\ell$ be two Hecke eigenforms. Assume $\langle f,g\rangle=0$ if $k=\ell$.  Then  we have
  \[
    \frac{2\pi^2}{k+\ell-1} \sum_{h\in H_{k+\ell}} \frac{
    L(1/2,f\times g\times  h )}{L(1,\sym^2 h)}
    = 2 \frac{L(1,\sym^2 f) L(1,\sym^2 g)}{\zeta(2)} (1+ O((\log (k+\ell))^{-1/4+\varepsilon })),
  \]
  as $\max(k,\ell)$ goes to infinity.
\end{theorem}

\begin{remark}
  Usually, one may expect to prove sharp upper bounds for moments of $L$-functions under GRH.
  It seems that no general method is known to prove asymptotic formulas for moments of $L$-functions.
  The other setting we know with our feature is Soundararajan and Young \cite{SY}.
\end{remark}

As an immediate corollary, we have the following nonvanishing result for the triple product central $L$-values.

\begin{corollary}\label{cor:nonvanishing}
  Assume GRH and GRC. Let $f,g$ be the same as in Theorem \ref{thm:moment}. There exist at least $(k+\ell)^{1-o(1)}$ Hecke eigenforms $h\in H_{k+\ell}$ such that $L(1/2,f\times g\times h)\neq 0$.
\end{corollary}

Under GRH, we have $L(1/2,f\times g\times h)=(k+\ell)^{o(1)}$ and $(\log\log k)^{-1} \ll  L(1, \sym^2 f) \ll (\log\log k)^{3}$ (see e.g. \cite[Theorem 3]{LW}). So Corollary \ref{cor:nonvanishing} follows from Theorem \ref{thm:moment}.
One may prove a better quantitative nonvanishing result by considering the second moment of the triple product  $L$-functions under GRH.

\subsection{Plan for this paper}
The rest of this paper is organized as follows.
In \S \ref{sec:preliminaries}, we give some lemmas on sums of Fourier coefficients and on $L$-functions.
In \S \ref{sec:joint-mass}, we use the spectral method and Watson's formula to prove Theorem \ref{thm:221} by assuming Propostion \ref{prop:u_j}.
In \S \ref{sec:fractional-moment}, a variant of Soundararajan's method is applied to give an almost sharp upper bound of certain fractional moment of $L$-functions, which proves Propostion \ref{prop:u_j}.
In \S \ref{sec:higher-decorrelation}, we give a sketch proof of Theorem \ref{thm:220}.
Finally, in \S \ref{sec:moment}, we use Theorem \ref{thm:221} to prove Theorem \ref{thm:moment}.
In the appendix, we give a heuristic argument of Theorem \ref{thm:moment}.

\medskip
\textbf{Notation.}
Throughout the paper, $\varepsilon$ is an arbitrarily small positive number;
all of them may be different at each occurrence.
As usual, $e(x)=e^{2\pi i x}$.
We use $y\asymp Y$ to mean that $c_1 Y\leq |y|\leq c_2 Y$ for some positive constants $c_1$ and $c_2$.
The symbol $f \ll g$ (or $g \gg f$) means that $|f| \le C g$
for some constant $C>0$. 
The implied constant may depend on $\varepsilon$.


\section{Preliminaries}\label{sec:preliminaries}


\subsection{The Kuznetsov trace formula}

Let $\mathcal{B}= \{u_j\}$ be an orthonormal basis of the space of Maass cusp forms for $\SL(2,\mathbb{Z})$,
where $u_j$ is a Hecke--Maass cusp form with the spectral parameter $t_j$.
The Fourier expansion of $u_j$ is given by
\[
  u_j(z) = \rho_j(1) \sum_{n\neq0} \lambda_j(n) |n|^{-1/2} W_{s_j}(nz) ,
\]
where $W_s(z)= 2\sqrt{|y|} K_{s-1/2}(2\pi |y|) e(x)$ is the Whittaker function and $s_j=1/2+it_j$.
Here $\lambda_j(n)$ is the $n$-th Hecke eigenvalue of $u_j$.
We may further assume that $u_j$ is even or odd according to $u_j(-\bar{z})=u_j(z)$ or $u_j(-\bar{z})=-u_j(z)$.
The symmetric square $L$-function of $u_j$ is defined by
\[
  L(s,\sym^2 u_j) = \zeta(2s) \sum_{n\geq1} \frac{\lambda_j(n^2)}{n^s}, \quad \Re(s)>1.
\]
Let  $h(t)$ be an even function satisfying that $h(t)$ is holomorphic in $|\Im(t)|\leq 1/2+\varepsilon$ and $h(t)\ll (1+|t|)^{-2-\varepsilon}$ in this strip.
We have the following Kuznetsov trace formula (see \cite[Eq. (3.17)]{CI})
\begin{multline}\label{eqn:KTF}
  {\sum_{j\geq1}}^{\rm even} \frac{\lambda_j(m)\overline{\lambda_j(n)}}{L(1,\sym^2 u_j)} h(t_j)
  + \frac{1}{4\pi} \int_{\mathbb{R}} \frac{\eta_t(m) \eta_t(n) }{|\zeta(1+2it)|^2} h(t)  \dd t
  \\
  = \frac{1}{2} \delta_{m,n} H + \frac{1}{2} \sum_{\pm}\sum_{c\geq1} \frac{S(n,\pm m,c)}{c} H^{\pm} \left(\frac{4\pi \sqrt{mn}}{c}\right),
\end{multline}
for $m,n\geq1$, where $\delta_{m,n}$ is 1 if $m=n$ and is 0 otherwise, and
\begin{align*}
  H & = \frac{1}{2\pi^2} \int_{0}^{\infty} h(t) \tanh(\pi t) t \dd t, \\
  H^+(x) & = \frac{i}{2\pi} \int_{-\infty}^{\infty} J_{2it}(x)\frac{h(t)t}{\cosh(\pi t)} \dd t, \\
  H^-(x) & = \frac{1}{\pi^2} \int_{-\infty}^{\infty} K_{2it}(x)\sinh(\pi t)h(t)t \dd t,
\end{align*}
and $J_\nu(x)$ and $K_\nu(x)$ are the standard $J$-Bessel function and $K$-Bessel function respectively.
Here $\eta_t(n)=\sum_{ab=n}(a/b)^{it}$ and
$S(m,n,c)$ is the classical Kloosterman sum. We have the following well-known  Weil bound (see e.g. \cite[Corollary 11.12]{IwaniecKowalski2004analytic})
\[
  S(m,n,c)=\sum_{\substack{d \ (\mod c) \\ (d,c)=1 } }
  e\left(\frac{md+n\bar{d}}{c}\right) \leq c^{1/2} (m,n,c)^{1/2} \tau(c),
\]
where $\bar{d}$ means $\bar{d}d\equiv 1\ \mod c$ and $\tau(c)=\sum_{d\mid c}1$ is the divisor function.

We will need a standard mean value estimate following from the Kuznetsov trace formula, whose proof we include for completeness.

\begin{lemma}\label{lemma:KTF}
  Let $T$ be large, and $m,n$ be positive integers with $mn\leq T^2$. Then we have
  \[
    {\sum_{j\geq1}}^{\rm even} \frac{\lambda_j(m) \lambda_j(n)}{L(1,\sym^2 u_j)} e^{-t_j^2/T^2}
    \ll \delta_{m,n} T^2 + (mn)^{1/2} T^{1+\varepsilon}.
  \]
\end{lemma}

\begin{remark}
  One can prove stronger result by dealing with oscillatory integrals. However, the above lemma is good enough for our application.
\end{remark}

\begin{proof}
  Let $h(t)=\exp(-t^2/T^2)$. By \eqref{eqn:KTF}, we have
  \begin{multline*}
  {\sum_{j\geq1}}^{\rm even} \frac{\lambda_j(m)\overline{\lambda_j(n)}}{L(1,\sym^2 u_j)} h(t_j)
  \\
  = \frac{1}{2} \delta_{m,n} H
  + \frac{1}{2} \sum_{\pm}\sum_{c\geq1} \frac{S(n,\pm m,c)}{c}
  H^{\pm} \left(\frac{4\pi \sqrt{mn}}{c}\right)
  -\frac{1}{4\pi} \int_{\mathbb{R}} \frac{\eta_t(m) \eta_t(n) }{|\zeta(1+2it)|^2} h(t)  \dd t.
\end{multline*}
  It is easy to see that $H\ll T^2$. By the standard bounds of $\zeta(1+2it)$, we can also estimate the Eisenstein series contribution trivially
  \[
    \frac{1}{4\pi} \int_{\mathbb{R}} \frac{\eta_t(m) \eta_t(n) }{|\zeta(1+2it)|^2} h(t)  \dd t
    \ll \int_{\mathbb{R}} \tau(m)\tau(n) (\log T)^2 e^{-t^2/T^2} \dd t \ll T^{1+\varepsilon}.
  \]

  Now we estimate these terms with Bessel functions. Denote
  \[
    S^\pm = \sum_{c\geq1} \frac{S(n,\pm m,c)}{c} H^{\pm} \left(\frac{4\pi \sqrt{mn}}{c}\right).
  \]
  For $H^+$, the same proof as in \cite[Lemma 7.1]{Young} gives
  \[
    H^+(x) \ll T x , \quad \textrm{for } x\ll T.
  \]
  For $H^-$, the same proof as in \cite[Lemma 7.2]{Young} gives
  \[
    H^-(x) \ll T^{1+\delta} x^{1-\delta}, \quad \textrm{for } x\ll T, \ \delta\in(0,1).
  \]
  So for $mn\leq T^2$ we have
  \[
    S^\pm \ll \sum_{c\geq1} \frac{(m,n,c)^{1/2}}{c^{1/2-\varepsilon}} T^{1+\delta} \left(\frac{\sqrt{mn}}{c}\right)^{1-\delta}
    \ll T^{1+\delta+\varepsilon} (mn)^{1/2-\delta/2},
  \]
  for any $\delta\in(0,1/2)$.  Taking $\delta=\varepsilon$, we complete the proof.
\end{proof}

\subsection{$L$-values to Dirichlet series}

Let $\alpha_j(p)$ and $\beta_j(p)$ denote the Satake parameters for $u_j$ at $p$.
We have $\alpha_j(p)+\beta_j(p)=\lambda_j(p)$, $\alpha_j(p)\beta_j(p)=1$, and
\[
  L(s,u_j) = \sum_{n\geq1} \frac{\lambda_j(n)}{n^s}
  = \prod_{p} \prod_{i\in\{\pm1\}} \left(1-\frac{\alpha_j(p)^i}{p^s}\right)^{-1},
  \quad \Re(s)>1.
\]
Let
\[
  \Lambda_j(p^b) = \alpha_j(p)^b +\alpha_j(p)^{-b}.
\]
Note that $L(1/2,u_j)\geq0$ (see \cite{KatokSarnak}).
We will use the following lemma.
\begin{lemma}\label{lemma:logL1}
  Assume GRH.  Then for $x>10$, we have
  \[
    \log L(1/2,u_j) \leq \sum_{p^b \leq x} \frac{\Lambda_j(p^b) }{b p^{b(1/2+\frac{1}{\log x})}} \frac{\log \frac{x}{p^b}}{\log x}
    + O\left( \frac{\log t_j}{\log x} + 1 \right).
  \]
\end{lemma}

\begin{proof}
  This follows from \cite[Theorem 2.1]{Chandee}.
\end{proof}

Let $f\in H_{k}$. Let $\sym^2f$ be the symmetric square lift of $f$. Then we have
\[
  L(s,\sym^2f\times u_j) = \sum_{n\geq1} \frac{\lambda_{\sym^2f\times u_j}(n)}{n^s}
  = \prod_{p} \prod_{i_1\in\{0,\pm1\}} \prod_{i_2\in\{\pm1\}}\left(1-\frac{\alpha_f(p)^{2i_1} \alpha_j(p)^{i_2}}{p^s}\right)^{-1},
\]
if $\Re(s)>1$.
Let
\[
  \Lambda_{\sym^2 f \times u_j}(p^c)
  = \sum_{i_1\in\{0,\pm1\}} \sum_{i_2\in\{\pm1\}} \alpha_f(p)^{2ci_1}\alpha_j(p)^{ci_2}  .
\]
Note that $L(1/2,\sym^2 f \times u_j)\geq0$ (see \cite{Lapid}). We will use the following lemma.
\begin{lemma}\label{lemma:logL2}
  Assume GRH. Let $f\in H_k$ and $u_j\in \mathcal{B}$. Assume $t_j\leq k$. Then for $x>10$, we have
  \[
    \log L(1/2,\sym^2 f \times u_j)
    \leq
    \sum_{p^c \leq x} \frac{\Lambda_{\sym^2 f \times u_j}(p^c) }{c p^{c(1/2+\frac{1}{\log x})}} \frac{\log \frac{x}{p^c}}{\log x}
    + O\left( \frac{\log k}{\log x} + 1 \right).
  \]
\end{lemma}

\begin{proof}
  This follows from \cite[Theorem 2.1]{Chandee}.
\end{proof}

In particular, we have
\begin{equation}\label{eqn:Lambda1}
  \Lambda_j(p) = \lambda_j(p), \quad
  \Lambda_j(p^2) = \lambda_j(p^2)-1
\end{equation}
and
\begin{equation}\label{eqn:Lambda2}
  \begin{split}
   \Lambda_{\sym^2 f \times u_j}(p)
  & = \sum_{i_1\in\{0,\pm1\}} \sum_{i_2\in\{\pm1\}} \alpha_f(p)^{2i_1}\alpha_j(p)^{i_2}
  = \lambda_f(p^2)\lambda_j(p), \\
  \Lambda_{\sym^2 f \times u_j}(p^2)
  & = \sum_{i_1\in\{0,\pm1\}} \sum_{i_2\in\{\pm1\}} \alpha_f(p)^{4i_1}\alpha_j(p)^{2i_2}
  = (\lambda_f(p^4)-\lambda_f(p^2)+1) (\lambda_j(p^2)-1).
  \end{split}
\end{equation}

\subsection{Sums over primes}

The following lemmas show that the sums over prime powers above may be restricted just to primes.

\begin{lemma}\label{lemma:PNT1}
  Assume GRH and GRC. Let $f\in H_k$ and $u_j\in \mathcal{B}$, and assume $t_j\leq k$. For $x>10$, we have
  \[
    \sum_{p\leq x} \frac{\lambda_j(p^2)}{p} = O(\log \log \log (t_j+x)) ,
  \]
  \[
    \sum_{p\leq x} \frac{\lambda_f(p^2)\lambda_j(p^2)}{p} = O(\log \log \log (k+x)) ,
  \]
  and
  \[
    \sum_{p\leq x} \frac{\lambda_f(p^4)\lambda_j(p^2)}{p} = O(\log \log \log (k+x)).
  \]
\end{lemma}

\begin{lemma}\label{lemma:PNT2}
  Assume GRH. Let $f\in H_k$ and $g\in H_\ell$, and $f\neq g$ if $k=\ell$.  Assume $ \ell^{1/2}\leq k\leq \ell$. For $x>10$, we have
  \[
    \sum_{p\leq x} \frac{\lambda_f(p^2)}{p} = O(\log \log \log (k+x)) ,
  \]
  \[
    \sum_{p\leq x} \frac{\lambda_f(p^4)}{p} = O(\log \log \log (k+x)) ,
  \]
  and
  \[
    \sum_{p\leq x} \frac{\lambda_{f}(p^2)\lambda_{g}(p^2)}{p} = O(\log \log \log (k+x)) .
  \]
\end{lemma}

The proofs of the above lemmas are standard. See e.g. \cite[Lemma 6.5]{HuangLester} and \cite[Lemma 4.3]{HHL}.

\section{Joint mass distribution: Proof of Theorem \ref{thm:221}}\label{sec:joint-mass}

Assume $k\leq \ell$.
Let $u_0=\vol(\Gamma\backslash\mathbb{H})^{-1/2}$,
$\{u_j\}_{j\geq1}$ be an $L^2$-normalized Hecke basis of the space of cuspidal Maass forms,
and $E_t(z)=E(z,1/2+it)$ the Eisenstein series.
By Parseval's identity, we have
\begin{equation}\label{eqn:period2moment}
  \langle |F|^2, |G|^2 \rangle
  = \langle |F|^2, u_0 \rangle \langle u_0, |G|^2 \rangle
  + \sum_{j\geq1} \langle |F|^2, u_j \rangle \langle u_j, |G|^2 \rangle
  + \frac{1}{4\pi} \int_{\mathbb{R}} \langle |F|^2, E_t \rangle \langle E_t, |G|^2 \rangle \dd t.
\end{equation}
The constant term contribution is
\begin{equation}\label{eqn:constantterm}
  \langle |F|^2, u_0 \rangle \langle u_0, |G|^2 \rangle = \frac{1}{\vol(\Gamma\backslash\mathbb{H})} \|f\|_2^2 \|g\|_2^2 = \vol(\Gamma\backslash\mathbb{H}).
\end{equation}

\subsection{Some calculations of Gamma functions}

In order the deal with the contribution from the cusp forms and the Eisenstein series, we will need to prove upper bounds for $|\Gamma(k-1/2+it)|/\Gamma(k)$. This is standard but a little tricky since both $k$ and $t$ can be large. So we give some details.
By Stirling's formula, we have
\begin{align*}
  \log |\Gamma(k-1/2+it)|^2 & = 2\Re (k-1/2+it)\log (k-1/2+it) \\
  & \hskip 3cm -(2k-1)- \log |k-1/2+it| + O(1) \\
  & = (k-1) \log ((k-1/2)^2+t^2) - 2t \theta -2k+ O(1).
\end{align*}
where $\theta={\rm Arg}(k-1/2+it)\in(-\pi/2,\pi/2)$.
Note that $\log \Gamma(k)^2 = 2 k\log k-2k - \log k + O(1)$. We have
\begin{align*}
  \log \frac{|\Gamma(k-1/2+it)|^2}{\Gamma(k)^2} &
   =  (k-1) \log ((k-1/2)^2+t^2) - 2t \theta  - 2k\log k + \log k + O(1).
\end{align*}
Let $t=\alpha (k-1/2)$, and assume $\alpha\geq0$. Then we have
\begin{align}
  \log \frac{|\Gamma(k-1/2+it)|^2}{\Gamma(k)^2}
   & = -\log k + (k-1) \log (1+\alpha^2) - 2\alpha k {\rm Arctan} (\alpha )  + O(1) \label{eqn:Stirling1} \\
   & \leq \Big(\frac{\log (1+\alpha^2)}{\alpha} - {\rm Arctan} (\alpha )\Big) t - {\rm Arctan} (\alpha )t  + O(1)  \label{eqn:Stirling2}.
\end{align}

We first treat the case when $t$ is very large.
If $\alpha\geq 100$, then we have $\theta={\rm Arctan} (\alpha )\geq 1.56$ and
$\frac{\log (1+\alpha^2)}{\alpha} - {\rm Arctan} (\alpha )<0$.
So by \eqref{eqn:Stirling2}, for $|t|\geq 100 k$ we have
\begin{equation}\label{eqn:Gamma-t}
  \frac{|\Gamma(k-1/2+it)|}{\Gamma(k)}  \ll \exp(-|t|/2).
\end{equation}

Define $\phi(\alpha) = \log (1+\alpha^2) - 2\alpha  {\rm Arctan} (\alpha )$.
Since $\phi'(\alpha) = - 2 {\rm Arctan} (\alpha )\leq 0$ for $\alpha\geq0$,
we have $\phi(\alpha)\leq \phi(\alpha_0)\leq \phi(0)=0$ for any $\alpha\geq \alpha_0\geq0$.
Taking $\alpha_0= k^{-1/3}$, we get
\begin{align} \label{eqn:Gamma-large}
  \frac{|\Gamma(k-1/2+it)|}{\Gamma(k)}  \ll \exp(-k^{1/3}/4),
\end{align}
for all $|t|\geq k^{2/3}$.

For $|t|\leq k^{2/3}$, by \eqref{eqn:Stirling1} we obtain
\[
    \log \frac{|\Gamma(k-1/2+it)|^2}{\Gamma(k)^2}
   = -\log k -\alpha^2 k   + O(1)
   = -t^2/k - \log k + O(1).
\]
Hence we have
\begin{equation} \label{eqn:Gamma-small}
  \frac{|\Gamma(k-1/2+it)|}{\Gamma(k)}
   \ll \exp(-t^2/(2k)) \frac{1}{\sqrt{k}}.
\end{equation}
By \eqref{eqn:Gamma-large} and \eqref{eqn:Gamma-small} we know
\begin{equation}\label{eqn:Gamma-all}
  \frac{|\Gamma(k-1/2+it)|}{\Gamma(k)}
   \ll \frac{1}{\sqrt{k}},
\end{equation}
for all $t\in \mathbb{R}$.

\subsection{The Eisenstein series contribution}

Next, we estimate the contribution from the Eisenstein series.
By the Rankin--Selberg method, we have
\begin{equation*}
  |\langle |F|^2, E_t \rangle| \ll
  \frac{|\zeta(1/2+it) L(1/2+it,\sym^2 f)|} {L(1,\sym^2 f)  |\zeta(1+2it)| } \frac{|\Gamma(k-1/2+it)|}{\Gamma(k)}.
\end{equation*}
So we have
\begin{multline}\label{eqn:RS}
  \frac{1}{4\pi} \int_{\mathbb{R}} \langle |F|^2, E_t \rangle \langle E_t, |G|^2 \rangle \dd t
  \ll \int_{\mathbb{R}}
  \frac{|\zeta(1/2+it)|^2 |L(1/2+it,\sym^2 f)L(1/2+it,\sym^2 g)|} {L(1,\sym^2 f)L(1,\sym^2 g)  |\zeta(1+2it)|^2 }
  \\ \cdot
  \frac{|\Gamma(k-1/2+it)|}{\Gamma(k)} \frac{|\Gamma(\ell-1/2+it)|}{\Gamma(\ell)}
  \dd t.
\end{multline}

\subsubsection{The case $k\leq \ell^\varepsilon$}

By \eqref{eqn:RS}, \eqref{eqn:Gamma-t}, and \eqref{eqn:Gamma-all}, we have
\begin{multline*}
  \frac{1}{4\pi} \int_{\mathbb{R}} \langle |F|^2, E_t \rangle \langle E_t, |G|^2 \rangle \dd t \\
  \ll
  \frac{1}{\sqrt{k\ell}} \int_{|t|\leq 100 k}
  \frac{|\zeta(1/2+it)|^2 |L(1/2+it,\sym^2 f)L(1/2+it,\sym^2 g)|} {L(1,\sym^2 f)L(1,\sym^2 g)  |\zeta(1+2it)|^2 } \dd t
  \\
  + \frac{1}{\sqrt{k\ell}} \int_{100 k \leq |t|\leq \ell^{1/2+\varepsilon}}
  \frac{|\zeta(1/2+it)|^2 |L(1/2+it,\sym^2 f)L(1/2+it,\sym^2 g)|} {L(1,\sym^2 f)L(1,\sym^2 g)  |\zeta(1+2it)|^2 } e^{-|t|/2} \dd t
  \\
  + O(\ell^{-2024}).
\end{multline*}
Assume GRH. Then for $|t|\leq \ell^{1/2+\varepsilon}$ we have
\[
  \frac{|\zeta(1/2+it)|^2 |L(1/2+it,\sym^2 f)L(1/2+it,\sym^2 g)|} {L(1,\sym^2 f)L(1,\sym^2 g)  |\zeta(1+2it)|^2 }
  \ll \ell^\varepsilon.
\]
Hence we get
\begin{equation}\label{eqn:ES<<1}
  \frac{1}{4\pi} \int_{\mathbb{R}} \langle |F|^2, E_t \rangle \langle E_t, |G|^2 \rangle \dd t
  \ll \ell^{-1/2+\varepsilon}k +
  \frac{\ell^\varepsilon}{\sqrt{k\ell}} \int_{|t|\leq \ell^{1/2+\varepsilon}} e^{-|t|/2} \dd t
  + O(\ell^{-2024})
  \ll \ell^{-1/2+\varepsilon}.
\end{equation}

\subsubsection{The case $\ell^\varepsilon \leq k\leq \ell$}

By \eqref{eqn:RS}, \eqref{eqn:Gamma-large}, and \eqref{eqn:Gamma-small}, we have
\begin{multline*}
  \frac{1}{4\pi} \int_{\mathbb{R}} \langle |F|^2, E_t \rangle \langle E_t, |G|^2 \rangle \dd t \\
  \ll
  \frac{1}{\sqrt{k\ell}} \int_{|t|\leq k^{1/2+\varepsilon}}
  \frac{|\zeta(1/2+it)|^2 |L(1/2+it,\sym^2 f)L(1/2+it,\sym^2 g)|} {L(1,\sym^2 f)L(1,\sym^2 g)  |\zeta(1+2it)|^2 } \dd t
  \\
  + O(\ell^{-2024}).
\end{multline*}
Hence we get
\begin{equation}\label{eqn:ES<<2}
  \frac{1}{4\pi} \int_{\mathbb{R}} \langle |F|^2, E_t \rangle \langle E_t, |G|^2 \rangle \dd t
  \ll
  \frac{\ell^\varepsilon}{\sqrt{k\ell}} \int_{|t|\leq k^{1/2+\varepsilon}}   \dd t
  + O(\ell^{-2024})
  \ll \ell^{-1/2+\varepsilon}.
\end{equation}

\subsection{The cusp form contribution}
Finally, we need to bound the contribution from the cusp forms. For $u_j$ odd, we have $\langle |F|^2,u_j\rangle=0$.
For $u_j$ even, by Watson's formula  we have
\begin{align*}
  |\langle |F|^2, u_j\rangle|^2 & \ll \frac{\Lambda(1/2,u_j) \Lambda(1/2,\sym^2 f\times u_j)}{\Lambda(1,\sym^2 f)^2\Lambda(1,\sym^2 u_j)} \\
  & \ll \frac{L(1/2,u_j) L(1/2,\sym^2 f\times u_j)}{L(1,\sym^2 f)^2 L(1,\sym^2 u_j)} \frac{|\Gamma(k-1/2+it_j)|^2}{\Gamma(k)^2}.
\end{align*}
So we get
\begin{multline*}
  \sum_{j\geq1} \langle |F|^2, u_j \rangle \langle u_j, |G|^2 \rangle
  \ll  {\sum_{j\geq1}}^{\rm even}
  \frac{L(1/2,u_j) L(1/2,\sym^2 f\times u_j)^{1/2}L(1/2,\sym^2 g\times u_j)^{1/2}}{L(1,\sym^2 f)L(1,\sym^2 g) L(1,\sym^2 u_j)}
  \\ \cdot \frac{|\Gamma(k-1/2+it_j)|}{\Gamma(k)} \frac{|\Gamma(\ell-1/2+it_j)|}{\Gamma(\ell)}.
\end{multline*}

\subsubsection{The case $k\leq \ell^\varepsilon$}

Assume $k\leq \ell^{\varepsilon}$.  Then by \eqref{eqn:Gamma-t} and \eqref{eqn:Gamma-all} we have
\begin{multline*}
  \sum_{j\geq1} \langle |F|^2, u_j \rangle \langle u_j, |G|^2 \rangle \\
  \ll \frac{1}{\sqrt{k\ell}} \sum_{t_j\leq 100 k \ell^\varepsilon}
  \frac{L(1/2,u_j) L(1/2,\sym^2 f\times u_j)^{1/2}L(1/2,\sym^2 g\times u_j)^{1/2}}{L(1,\sym^2 f)L(1,\sym^2 g) L(1,\sym^2 u_j)} \\
  + \frac{1}{\sqrt{\ell}} \sum_{100 k \ell^\varepsilon \leq t_j\leq \ell^{1/2+\varepsilon}}
  \frac{L(1/2,u_j) L(1/2,\sym^2 f\times u_j)^{1/2}L(1/2,\sym^2 g\times u_j)^{1/2}}{L(1,\sym^2 f)L(1,\sym^2 g) L(1,\sym^2 u_j)} \exp\left(-\frac{|t_j|}{2}\right) \\
  + O(\ell^{-2024}).
\end{multline*}
Under GRH, for $t_j\leq \ell^{1/2+\varepsilon}$ we know
\[
  \frac{L(1/2,u_j) L(1/2,\sym^2 f\times u_j)^{1/2}L(1/2,\sym^2 g\times u_j)^{1/2}}{L(1,\sym^2 f)L(1,\sym^2 g) L(1,\sym^2 u_j)}
  \ll \ell^\varepsilon.
\]
Hence we get
\begin{equation*}
  \sum_{j\geq1} \langle |F|^2, u_j \rangle \langle u_j, |G|^2 \rangle
  \ll \ell^{-1/2+\varepsilon}.
\end{equation*}

\subsubsection{The case $\ell^\varepsilon \leq k\leq \ell^{1-\varepsilon}$}

By \eqref{eqn:Gamma-large} and \eqref{eqn:Gamma-small}, we have
\begin{multline}\label{eqn:cusp<<}
  \sum_{j\geq1} \langle |F|^2, u_j \rangle \langle u_j, |G|^2 \rangle \\
  \ll \frac{1}{\sqrt{k\ell}} \;\ \  \sideset{}{^{\rm even}}\sum_{t_j\leq k^{1/2+\varepsilon}}
  \frac{L(1/2,u_j) L(1/2,\sym^2 f\times u_j)^{1/2}L(1/2,\sym^2 g\times u_j)^{1/2}}{L(1,\sym^2 f)L(1,\sym^2 g) L(1,\sym^2 u_j)} \exp\left(-\frac{t_j^2}{2k}\right) \\
  + O(\ell^{-2024}).
\end{multline}
Assume GRH. Then we have
\begin{equation*}
  \sum_{j\geq1} \langle |F|^2, u_j \rangle \langle u_j, |G|^2 \rangle
  \ll \frac{\ell^\varepsilon}{\sqrt{k\ell}} \sum_{t_j\leq k^{1/2+\varepsilon}} 1 + O(\ell^{-2024})
  \ll k^{1/2+\varepsilon/4}\ell^{-1/2} \ll \ell^{-\varepsilon}.
\end{equation*}

\subsubsection{The case $\ell^{1-\varepsilon} \leq k\leq \ell$}

We now deal with this most complicated case.
The inequality \eqref{eqn:cusp<<} still holds. But a trivial estimate will not suffice.
We will need the following proposition, the proof of which will be given in next section.

\begin{proposition}\label{prop:u_j}
  Assume GRH and GRC.  Let $T$ be a large number.  Let $k,\ell$ be even integers with $T^{3/2} \leq k\leq T^{100}$ and $T^{3/2} \leq \ell\leq T^{100}$.  Let $f\in H_k$ and $g\in H_\ell$, and $f\neq g$ if $k=\ell$. Then we have
  \begin{multline*}
    {\sum_{j\geq1}}^{\rm even}
  \frac{L(1/2,u_j)^{l_1} L(1/2,\sym^2 f\times u_j)^{l_2}L(1/2,\sym^2 g\times u_j)^{l_3}}{L(1,\sym^2 u_j)} \exp\left(-\frac{t_j^2}{T^2}\right) \\
  \ll T^2 (\log T)^{\sum_{i=1}^{3} \frac{l_i(l_i-1)}{2}+\varepsilon}, 
  \end{multline*}
   where the implied constant may depend on $l_1,l_2,l_3$ and $\varepsilon$.
\end{proposition}

Under GRH, for $f\in H_k$ we have $(\log\log k)^{-1} \ll  L(1, \sym^2 f) \ll (\log\log k)^{3}$ (see \cite[Theorem 3]{LW}).
  By \eqref{eqn:cusp<<} and Proposition \ref{prop:u_j} with $l_1=1$, $l_2=l_3=1/2$, and $T^2=2k$, we get
  \begin{equation*}
    \sum_{j\geq1} \langle |F|^2, u_j \rangle \langle u_j, |G|^2 \rangle  
    \ll \frac{1}{\sqrt{k\ell}} \cdot k (\log k)^{-1/4+\varepsilon}
     \ll \sqrt{\frac{k}{\ell}} (\log k)^{-1/4+\varepsilon}.
  \end{equation*}
  This completes the proof of Theorem \ref{thm:221}.

\section{Fractional moments of $L$-functions}\label{sec:fractional-moment}

In this section, we will prove Proposition \ref{prop:u_j} by using Soundararajan's method \cite{soundararajan2009moments}. See also \cite[\S6]{HuangLester} and \cite{HHL}.

\subsection{A mean value estimate}

\begin{lemma}\label{lemma:average}
  Let $r\in \mathbb{N}$. Then for $x\leq T^{1/(10 r)}$ and real numbers $a(p)$ with $a(p)\ll p^{1/2}$, we have that
  \[
    \frac{1}{T^2} \sum_{j\geq1} \frac{1}{L(1,\sym^2 u_j)} \left( \sum_{p\leq x} \frac{a(p) \lambda_j(p)}{p^{1/2}} \right)^{2r}   e^{-t_j^2/T^2} \ll
    \frac{(2r)!}{2^r  r!} \left( \sum_{p\leq x} \frac{|a(p)|^2}{p} \right)^{r}
    +   T^{-1/3}.
  \]
\end{lemma}

\begin{proof}
  Write
  \[
    \left( \sum_{p\leq x} \frac{a(p) \lambda_j(p)}{p^{1/2}} \right)^{2r}
    = \sum_{n\leq x^{2r}} \frac{a_{2r,x}(n)}{n^{1/2}} \prod_{n=\prod_{i=1}^{s} p_i^{\alpha_i}} \lambda_j(p_i)^{\alpha_i},
  \]
  where $a_{2r,x}(n)=0$ unless $n=\prod_{i=1}^{s} p_i^{\alpha_i}$ with $\sum_{i=1}^{s} \alpha_i=2r$ and distinct primes $p_i$ all below $x$, in which case we have $a_{2r,x}(n)=\binom{2r}{\alpha_1,\ldots,\alpha_s} \prod_{i=1}^{s} a(p_i)^{\alpha_i}$.

  By the Hecke relations and some computations, we know (see e.g. \cite[Lemma 7.1]{LL})
  \[
    \lambda_j(p)^{\alpha} = \sum_{l=0}^{\alpha/2} \frac{\alpha! (2l+1)}{(\alpha/2-l)! (\alpha/2+l+1)!} \lambda_j(p^{2l}),
  \]
  if $2\mid \alpha$; and
  \[
    \lambda_j(p)^{\alpha} =\sum_{l=0}^{(\alpha-1)/2} \frac{\alpha! (2l+2)}{((\alpha-1)/2-l)! ((\alpha+3)/2+l)!} \lambda_j(p^{2l+1}),
  \]
  if $2\nmid \alpha$.
  So we can write  $\lambda_j(p)^{\alpha}=\sum_{\beta=0}^{\alpha} b_{\alpha,\beta} \lambda_{j}(p^{\beta})$, where
  \[
    b_{\alpha,\beta} = \left\{
    \begin{array}{ll}
      \frac{\alpha! (\beta+1)}{(\frac{\alpha-\beta}{2})! (\frac{\alpha+\beta}{2}+1)!}, & \textrm{if } \alpha \equiv \beta \ \mod 2,\ 0\leq \beta\leq \alpha, \\
      0, & \textrm{otherwise.}
    \end{array}
    \right.
  \]
  Hence we obtain
  \[
    \left( \sum_{p\leq x} \frac{a_p \lambda_j(p)}{p^{1/2}} \right)^{2r}
    = \sum_{n\leq x^{2r}} \frac{a_{2r,x}(n)}{n^{1/2}} \sum_{m\mid n} b(m,n) \lambda_j(m),
  \]
  where $b(m,n)= \prod_{i=1}^{s} b_{\alpha,\beta}$ for $m=\prod_{i=1}^{s}p_i^{\beta_i}$.
  In particular, $b(1,n)=0$ unless $2\mid \alpha_i$ for all $i$,
  in which case we have $b(1,n)=\prod_{i=1}^{s} \frac{\alpha_i !}{(\alpha_i/2)! (\alpha_i/2+1)!}$.
  Note that we have
  \begin{equation}\label{eqn:b}
    0\leq b_{\alpha,\beta}\leq \binom{\alpha}{(\alpha-\beta)/2}\leq  2^{\alpha}, \quad
    0\leq b(m,n)\leq \prod_{i=1}^s 2^{\alpha_i}\leq n.
  \end{equation}

  Now we have
  \begin{multline*}
    S := \frac{1}{T^2} \sum_{j\geq1} \frac{1}{L(1,\sym^2 u_j)}  \left( \sum_{p\leq x} \frac{a_p \lambda_j(p)}{p^{1/2}} \right)^{2r}   e^{-t_j^2/T^2} \\
    =
    \sum_{n\leq x^{2r}} \frac{a_{2r,x}(n)}{n^{1/2}} \sum_{m\mid n} b(m,n)
    \frac{1}{T^2} \sum_{j\geq1} \frac{ \lambda_j(m)}{L(1,\sym^2 u_j)}  e^{-t_j^2/T^2}.
  \end{multline*}
  By Lemma \ref{lemma:KTF} we have
  \begin{align*}
    S & \ll \sum_{n\leq x^{2r}} \frac{|a_{2r,x}(n)|}{n^{1/2}} \sum_{m\mid n} b(m,n) (\delta_{1,m}+ m^{1/2} T^{-1+\varepsilon})
    \\
    & \ll \sum_{\substack{n\leq x^{2r}\\ 2\mid \alpha_i}} \frac{|a_{2r,x}(n)|}{n^{1/2}} \prod_{i=1}^{s} \frac{\alpha_i !}{(\alpha_i/2)! (\alpha_i/2+1)!} +
    \sum_{n\leq x^{2r}} \frac{|a_{2r,x}(n)|}{n^{1/2}} \sum_{m\mid n} b(m,n)  m^{1/2} T^{-1+\varepsilon} .
  \end{align*}
  By \eqref{eqn:b} and writing $n=m^2$ in the first summation, we have
  \begin{align*}
    S & \ll \sum_{\substack{m=\prod_{i=1}^{s} p_i^{\beta_i} \\ p_i\leq x,\ \sum_{i=1}^{s}\beta_i=r}}
    \frac{|a_{2r,x}(m^2)|}{m} \prod_{i=1}^{s} \frac{(2\beta_i) !}{(\beta_i)! (\beta_i+1)!} +
    \sum_{n\leq x^{2r}}  |a_{2r,x}(n)|  n^{1+\varepsilon} T^{-1+\varepsilon}
    \\
    & \ll \sum_{\substack{m=\prod_{i=1}^{s} p_i^{\beta_i} \\ p_i\leq x,\ \sum_{i=1}^{s}\beta_i=r}}
    \binom{2r}{2\beta_1,\ldots,2\beta_s} \prod_{i=1}^{s} \frac{|a(p_i)|^{2\beta_i}}{p_i^{\beta_i}} \prod_{i=1}^{s} \frac{(2\beta_i) !}{(\beta_i)! (\beta_i+1)!} \\
    & \hskip 120pt  +
    \sum_{n\leq x^{2r}} \binom{2r}{\alpha_1,\ldots,\alpha_s} \prod_{i=1}^{s} |a(p_i)|^{\alpha_i}  x^{3r} T^{-1+\varepsilon}
  \end{align*}
  Note that $\binom{2r}{2\beta_1,\ldots,2\beta_s} \prod_{i=1}^{s} \frac{(2\beta_i) !}{(\beta_i)! (\beta_i+1)!} \leq \frac{(2r)!}{\prod_{i=1}^{s}(\beta_i)! (\beta_i+1)!} \leq \frac{(2r)!}{2^r r!} \binom{r}{\beta_1,\ldots,\beta_s}$.
  Here we have used the fact  $\frac{1}{\prod_{i=1}^{s}(\beta_i+1)!} \leq \frac{1}{2^r} $ for any $\sum_{i=1}^{s} \beta_i = r$.
  Hence
  \begin{align*}
    S & \ll \frac{(2r)!}{2^r r!} \sum_{\substack{m=\prod_{i=1}^{s} p_i^{\beta_i} \\ p_i\leq x,\ \sum_{i=1}^{s}\beta_i=r}}
    \binom{r}{\beta_1,\ldots,\beta_s} \prod_{i=1}^{s} \frac{|a(p_i)|^{2\beta_i}}{p_i^{\beta_i}}
    +   x^{3r} T^{-1+\varepsilon} \left( \sum_{p\leq x} |a(p)|\right)^{2r}
    \\
    & \ll \frac{(2r)!}{2^r r!} \left( \sum_{p\leq x} \frac{|a(p)|^2}{p} \right)^{r}
    +   x^{6r} T^{-1+\varepsilon} .
  \end{align*}
  Here we have used the assumption $|a(p)|\ll p^{1/2}$. Since $x\leq T^{1/(10r)}$, this completes the proof of the lemma.
\end{proof}

\subsection{Proof of Proposition \ref{prop:u_j}}

Define
\[
  \mathcal{L}(u_j) = \mathcal{L}(u_j;f,g;l_1,l_2,l_3) := L(1/2,u_j)^{l_1} L(1/2,\sym^2 f\times u_j)^{l_2}L(1/2,\sym^2 g\times u_j)^{l_3}
\]
and
\[
  \mathcal{B}_T(V) := \sum_{\substack{j\geq1 \\ \log \mathcal{L}(u_j)>V}} \frac{1}{L(1,\sym^2 u_j)} \exp\left(-\frac{t_j^2}{T^2}\right).
\]
Clearly, we have
\begin{align*}
  \sum_{\substack{j\geq1}} \frac{\mathcal{L}(u_j)}{L(1,\sym^2 u_j)} \exp\left(-\frac{t_j^2}{T^2}\right)
  & =  \sum_{\substack{j\geq1}} \frac{1}{L(1,\sym^2 u_j)} \exp\left(-\frac{t_j^2}{T^2}\right)
  \int_{-\infty}^{\log \mathcal{L}(u_j)} e^V \dd V  \\
  & = \int_{-\infty}^{\infty} e^V \mathcal{B}_T(V) \dd V.
\end{align*}
Define
\[
  \mu(T) := (-1/2+\varepsilon)(l_1+l_2+l_3) \log\log T.
\]
Then we have
\begin{align}\label{eqn:sum2integral}
  \sum_{\substack{j\geq1}} \frac{\mathcal{L}(u_j)}{L(1,\sym^2 u_j)} \exp\left(-\frac{t_j^2}{T^2}\right)
   & = e^{\mu(T)} \int_{-\infty}^{\infty} e^V \mathcal{B}_T(V+\mu(T)) \dd V.
\end{align}
To prove Proposition \ref{prop:u_j}, we need to estimate $\mathcal{B}_T(V+\mu(T))$.
By Lemmas \ref{lemma:logL1} and \ref{lemma:logL2}, for any $T^{3/2} \leq k\leq T^{100}$ and $T^{3/2} \leq \ell\leq T^{100}$, we know that $\log \mathcal{L}(u_j) \leq C \log T/\log\log T$ for some $C=C(l_1,l_2,l_3)>1$.
Note that we have the trivial bound $\mathcal{B}_T(V) \ll T^2$. So
\[
  e^{\mu(T)} \int_{-\infty}^{(\log\log T)^{2/3}} e^V \mathcal{B}_T(V+\mu(T)) \dd V
  \ll T^2 e^{\mu(T)+(\log\log T)^{2/3}}
  \ll  T^2 e^{\mu(T)} (\log T)^{\varepsilon}.
\]
Hence we only need to consider $(\log\log T)^{2/3} \leq V \leq C \log T/\log\log T$ in the integral in \eqref{eqn:sum2integral}.

By \eqref{eqn:Lambda1} and bounding the sum over $p^b\leq x$ with $b\geq3$, we obtain
\[
  \sum_{p^b \leq x} \frac{\Lambda_j(p^b) }{b p^{b(1/2+\frac{1}{\log x})}} \frac{\log \frac{x}{p^b}}{\log x}
  =
  \sum_{p \leq x} \frac{\lambda_j(p) }{p^{1/2+\frac{1}{\log x}}} \frac{\log \frac{x}{p}}{\log x}
  +
  \frac{1}{2} \sum_{p \leq \sqrt{x}} \frac{(\lambda_j(p^2)-1) }{p^{1+\frac{2}{\log x}}} \frac{\log \frac{x}{p^2}}{\log x}
  + O(1).
\]
By Lemma \ref{lemma:PNT1} and the prime number theorem, the second sum on the right hand side above equals
\[
  -\frac{1}{2} \log \log x + O(\log\log \log (t_j+x)).
\]
Similarly, by \eqref{eqn:Lambda2} and bounding the sum over $p^c\leq x$ with $c\geq3$, we obtain
\begin{multline*}
  \sum_{p^c \leq x} \frac{\Lambda_{\sym^2 f \times u_j}(p^c) }{c p^{c(1/2+\frac{1}{\log x})}} \frac{\log \frac{x}{p^c}}{\log x}
   \\
   =
  \sum_{p \leq x} \frac{\lambda_f(p^2)\lambda_j(p) }{p^{1/2+\frac{1}{\log x}}} \frac{\log \frac{x}{p}}{\log x}
  +
  \frac{1}{2} \sum_{p \leq \sqrt{x}} \frac{(\lambda_f(p^4)-\lambda_f(p^2)+1) (\lambda_j(p^2)-1)}{p^{1+\frac{2}{\log x}}} \frac{\log \frac{x}{p^2}}{\log x}
  + O(1).
\end{multline*}
By Lemmas \ref{lemma:PNT1} and \ref{lemma:PNT2} and the prime number theorem, the second sum on the right hand side above equals
\[
  -\frac{1}{2} \log \log x + O(\log\log \log (kt_j+x)).
\]
Hence for $T^{3/2} \leq k\leq T^{100}$, $T^{3/2} \leq \ell\leq T^{100}$, and $t_j\leq T^{1+\varepsilon}$, we have
\begin{multline}\label{eqn:logL-sums}
  \log \mathcal{L}(u_j)
  \leq
  \sum_{p \leq x} \frac{(l_1  + l_2 \lambda_f(p^2) + l_3 \lambda_g(p^2) ) \lambda_j(p)  }
  {p^{1/2+\frac{1}{\log x}}} \frac{\log \frac{x}{p}}{\log x}
  \\
  -\frac{1}{2} (l_1+l_2+l_3) \log \log x + O(\log\log\log T).
\end{multline}

For $2\leq y\leq x$, let
\[
  \mathcal{P}(u_j;x,y) := \sum_{p \leq y} \frac{(l_1  + l_2 \lambda_f(p^2) + l_3 \lambda_g(p^2) ) \lambda_j(p)  }
  {p^{1/2+\frac{1}{\log x}}} \frac{\log \frac{x}{p}}{\log x}
\]
and
\[
  \mathcal{A}_T(V;x) := \sum_{\substack{j\geq1 \\ \mathcal{P}(u_j;x,x)>V}} \frac{1}{L(1,\sym^2 u_j)} \exp\left(-\frac{t_j^2}{T^2}\right).
\]
Also, define
\[
  \sigma(T)^2 = (l_1^2+l_2^2+l_3^2) \log\log T.
\]
In order to estimate $\mathcal{B}_T(V)$, we will first prove bounds for $\mathcal{A}_T(V;x)$ as follows.

\begin{lemma}\label{lemma:A<<}
  Let $C\geq1$ be fixed and $\epsilon>0$ be sufficiently small. For $(\log\log T)^{2/3}\leq V \leq C\log T/\log\log T$, we have
  \[
    \mathcal{A}_T(V;T^{\frac{1}{\epsilon V}}) \ll T^2 \left( \exp\Big(-\frac{V^2}{2\sigma(T)^2}(1-2\epsilon)\Big) + \exp\Big(-\frac{\epsilon}{11} V\log V \Big) \right) + T^{5/3}.
  \]
\end{lemma}

\begin{proof}
  Let $x=T^{\frac{1}{\epsilon V}}$ and $z=x^{1/\log\log T}$. Also, let $V_1=(1-\epsilon) V$ and $V_2=\epsilon V$.
  If $\mathcal{P}(u_j;x,x)>V$, then
  \[
    \textrm{ i) } \ \mathcal{P}(u_j;x,z)> V_1 \qquad
    \textrm{or} \qquad
    \textrm{ ii) } \ \mathcal{P}(u_j;x,x)-\mathcal{P}(u_j;x,z)> V_2.
  \]

  We first deal with the case i).
  By Lemma \ref{lemma:average} we see that for any natural number $r\leq \frac{\epsilon V}{10}\log\log T$, we have
  \begin{align*}
    \sum_{\substack{j\geq1 \\ \mathcal{P}(u_j;x,z)>V_1}} \frac{1}{L(1,\sym^2 u_j)} \exp\left(-\frac{t_j^2}{T^2}\right)
    & < \frac{1}{V_1^{2r}} \sum_{j\geq1} \frac{1}{L(1,\sym^2 u_j)} \exp\left(-\frac{t_j^2}{T^2}\right)  \mathcal{P}(u_j;x,z)^{2r}
    \\
    & \ll \frac{T^2}{V_1^{2r}} \left( \frac{(2r)!}{2^r  r!} \left( \sum_{p\leq z} \frac{|a(p)|^2}{p} \right)^{r}
    +   T^{-1/3}\right),
  \end{align*}
  with $a(p)=\frac{(l_1  + l_2 \lambda_f(p^2) + l_3 \lambda_g(p^2) )}
  {p^{\frac{1}{\log x}}} \frac{\log \frac{x}{p}}{\log x}$.
  By Lemma \ref{lemma:PNT2} and the prime number theorem, we know
  \[
    \sum_{p\leq z} \frac{(l_1  + l_2 \lambda_f(p^2) + l_3 \lambda_g(p^2) )^2}{p}
    = (l_1^2 + l_2^2+l_3^2) \log \log z + O(\log\log\log (k+z)).
  \]
  So we have
  \begin{multline*}
    \sum_{\substack{j\geq1 \\ \mathcal{P}(u_j;x,z)>V_1}} \frac{1}{L(1,\sym^2 u_j)} \exp\left(-\frac{t_j^2}{T^2}\right)
    \\
     \ll \frac{T^2}{V_1^{2r}} \left( \frac{(2r)!}{2^r r!} \big( (l_1^2 + l_2^2+l_3^2) \log \log z + O(\log\log\log T) \big)^{r}
    +   T^{-1/3}\right) \\
    \ll \frac{T^2}{V_1^{2r}}  \frac{(2r)!}{2^r r!} \big( (1+o(1) \sigma(T)^2 )\big)^{r}
    +   T^{5/3}.
  \end{multline*}
  By Stirling's formula, we get
  \begin{equation*}
    \sum_{\substack{j\geq1 \\ \mathcal{P}(u_j;x,z)>V_1}} \frac{1}{L(1,\sym^2 u_j)} \exp\left(-\frac{t_j^2}{T^2}\right)
    \ll T^2 \left( \frac{2r(1+o(1)) \sigma(T)^2 }{e V_1^2 } \right)^{r}
    +   T^{5/3}.
  \end{equation*}
  Take
  \[
    r = \left\{ \begin{array}{ll}
          \lfloor \frac{V_1^2}{2\sigma(T)^2} \rfloor, & \textrm{if } (\log\log T)^{2/3} \leq V \leq \frac{\epsilon}{10} \sigma(T)^2 \log\log T, \\
          \lfloor \frac{\epsilon V}{10} \rfloor, & \textrm{if } \frac{\epsilon}{10} \sigma(T)^2 \log\log T < V \leq C\log T/\log\log T.
        \end{array}\right.
  \]
  Then we have
  \begin{multline*}
    \sum_{\substack{j\geq1 \\ \mathcal{P}(u_j;x,z)>V_1}} \frac{1}{L(1,\sym^2 u_j)} \exp\left(-\frac{t_j^2}{T^2}\right)
    \\
    \ll T^2 \left( \exp\Big(-\frac{V^2}{2\sigma(T)^2}(1-2\epsilon)\Big) + \exp\Big(-\frac{\epsilon}{11} V\log V \Big) \right)
    +   T^{5/3}.
  \end{multline*}

  Now we treat the case ii).
  As before, by Lemma \ref{lemma:average} we see that for  $r= \lfloor \frac{\epsilon V}{10}\rfloor$, we have
  \begin{align*}
    \sum_{\substack{j\geq1 \\ \mathcal{P}(u_j;x,x)- \mathcal{P}(u_j;x,z)>V_2}}
    & \frac{1}{L(1,\sym^2 u_j)} \exp\left(-\frac{t_j^2}{T^2}\right) \\
    & < \frac{1}{V_2^{2r}} \sum_{j\geq1} \frac{(\mathcal{P}(u_j;x,x)- \mathcal{P}(u_j;x,z))^{2r}}{L(1,\sym^2 u_j)} \exp\left(-\frac{t_j^2}{T^2}\right)
    \\
    & \ll \frac{T^2}{V_2^{2r}} \left( \frac{(2r)!}{2^r  r!} \left( \sum_{z< p\leq x} \frac{|a(p)|^2}{p} \right)^{r}
    +   T^{-1/3}\right).
  \end{align*}
  By Deligne's bound \cite{Deligne}, we have  $|a(p)| \leq l_1+3l_2+3l_3$. By the prime number theorem we get
  \begin{align*}
    & \sum_{\substack{j\geq1 \\ \mathcal{P}(u_j;x,x)- \mathcal{P}(u_j;x,z)>V_2}}
    \frac{1}{L(1,\sym^2 u_j)} \exp\left(-\frac{t_j^2}{T^2}\right) \\
    & \hskip 50pt \ll \frac{T^2}{V_2^{2r}} \left( \frac{(2r)!}{2^r  r!} \left( (l_1+3l_2+3l_3)^2 (\log\log x-\log \log z+O(1)) \right)^{r}
    +   T^{-1/3}\right) \\
    & \hskip 50pt \ll T^2   \left( \frac{2 r  (l_1+3l_2+3l_3)^2 (\log\log\log T + O(1))}{e \epsilon^2 V^2} \right)^{r}
    +   T^{5/3}
    \\
    & \hskip 50pt \ll T^2  \exp\left( -\frac{\epsilon}{11} V\log V \right)
    +   T^{5/3} .
  \end{align*}
  This completes the proof of the lemma.
\end{proof}

\begin{proof}[Proof of Proposition \ref{prop:u_j}]
  Let $x=T^{1/\epsilon V}$. For $(\log\log T)^{2/3} \leq V \leq (\log\log T)^{4}$, we have
  \[
    -\frac{1}{2} (l_1+l_2+l_3) \log\log x + O(\log\log\log T) \leq \mu(T).
  \]
  It follows from \eqref{eqn:logL-sums} that
  \[
    \mathcal{B}_T(V+\mu(T)) \leq \mathcal{A}_T(V(1-2\varepsilon);x)
  \]
  when $(\log\log T)^{2/3} \leq V \leq (\log\log T)^{4}$.
  For $V> (\log\log T)^4$, we have $V+\mu(T)= V(1+o(1))$, so the above inequality still holds.

  Now by \eqref{eqn:sum2integral} and Lemma \ref{lemma:A<<}, we have
  \begin{multline*}
    \sum_{\substack{j\geq1}} \frac{\mathcal{L}(u_j)}{L(1,\sym^2 u_j)} \exp\left(-\frac{t_j^2}{T^2}\right)
     \\
      \leq  e^{\mu(T)} \int_{(\log\log T)^{2/3}}^{C\frac{\log T}{\log\log T}} e^V \mathcal{A}_T(V(1-2\varepsilon);x) \dd V
     + O\left(T^2 e^{\mu(T)} (\log T)^{\varepsilon}\right)
     \\
      \ll
      T^2 e^{\mu(T)} \int_{(\log\log T)^{2/3}}^{C\frac{\log T}{\log\log T}} e^V
       \left( \exp\Big(-\frac{V^2}{2\sigma(T)^2}(1-\varepsilon)\Big) + \exp\Big(-\frac{\varepsilon}{11} V\log V \Big) \right)  \dd V
      \\
       + T^2 e^{\mu(T)} (\log T)^{\varepsilon}.
  \end{multline*}
  Note that $\int_{\mathbb{R}} \exp(-x^2/(2\sigma^2)+x) \dd x = \sqrt{2\pi} \sigma \exp(\sigma^2/2)$.
  Hence we have
  \begin{align*}
    \sum_{\substack{j\geq1}} \frac{\mathcal{L}(u_j)}{L(1,\sym^2 u_j)} \exp\left(-\frac{t_j^2}{T^2}\right)
     &  \ll
      T^2 e^{\mu(T)+\sigma(T)^2/2} (\log T)^\varepsilon
       + T^2 e^{\mu(T)} (\log T)^{\varepsilon} \\
     & \ll T^2 (\log T)^{\sum_{i=1}^{3} \frac{l_i(l_i-1)}{2}+\varepsilon}.
  \end{align*}
  This completes the proof of Proposition \ref{prop:u_j}.
\end{proof}

\section{Higher decorrelation: Proof of Theorem \ref{thm:220}} \label{sec:higher-decorrelation}

In this section, we give a sketch proof of Theorem \ref{thm:220}. By Parseval's identity, we have
\begin{equation}\label{Selberg spectral decomposition2}
  \langle f^2 , g^2\rangle
  =\sum_{h\in H_{2k}} \langle f^2 , h\rangle \langle h, g^2 \rangle .
\end{equation}
Watson's formula \cite{Watson} gives
\begin{equation}\label{Watson's formula}
  |\langle f^2 , h\rangle|^2
  \ll \frac{\pi^3}{2(2k-1)} \frac{L(1/2, h)L(1/2, \sym^2f \times h)}
    {L(1, \sym^2 f)^2 L(1, \sym^2h)}.
\end{equation}
Note that these $L$-values are nonnegative \cite{KohnenZagier,Lapid}. So we have
\begin{equation}
  \langle f^2 , g^2\rangle  \ll \frac{1}{k}
  \sum_{h\in H_{2k}}
  \frac{L(1/2, h) L(1/2, \sym^2f \times h)^{1/2} L(1/2, \sym^2g \times h)^{1/2}}
    {L(1, \sym^2 f) L(1, \sym^2 g) L(1, \sym^2h)}.
\end{equation}
We will use  the following proposition on the fractional moment of $L$-functions.
\begin{proposition}\label{prop:h}
  Let $f,g\in H_k$ with $\langle f,g\rangle=0$. For $l,l_1,l_2\geq0$, we have that
  \begin{align*}
   \frac{1}{k}
  \sum_{h\in H_{2k}} L(1/2, h)^l L(1/2, \sym^2f \times h)^{l_1}
  & L(1/2, \sym^2g \times h)^{l_2} \\
   & \ll  (\log k)^{\frac{l(l-1)}{2}+\frac{l_1(l_1-1)}{2} +\frac{l_2(l_2-1)}{2}+\varepsilon},
\end{align*}
where the implied constant may depend on $l,l_1,l_2$ and $\varepsilon$.
\end{proposition}

The proof of Proposition \ref{prop:h} is similar to Proposition \ref{prop:u_j}, so we do not give the details.
We only mention that in the proof one can use the following lemma on the Petersson's formula to replace Lemma \ref{lemma:KTF}.
\begin{lemma}
  Let $k$ be a large even integer. For two natural numbers $m$ and $n$ with $mn\leq k^2/10^4$, we have
  \[
    \frac{2\pi^2}{k-1}\sum_{f\in H_k} \frac{\lambda_f(m)\lambda_f(n)}{L(1,\sym^2 f)} = \delta_{m,n} + O(e^{-k}).
  \]
\end{lemma}

\begin{proof}
  This is \cite[Lemma 2.1]{RudnickSoundararajan}.
\end{proof}

\begin{proof}[Proof of Theorem \ref{thm:220}]
  Under GRH, for $f\in H_k$ we have $(\log\log k)^{-1} \ll  L(1, \sym^2 f) \ll (\log\log k)^{3}$ (see \cite[Theorem 3]{LW}).
  By Proposition \ref{prop:h} with $l=1$ and $l_1=l_2=1/2$, we get
  \begin{equation*}
    \langle f^2 , g^2\rangle  \ll (\log k)^{-1/4+\varepsilon}.
  \end{equation*}
  This completes the proof of Theorem \ref{thm:220}.
\end{proof}

%
%

\section{First moment of the triple product $L$-functions}\label{sec:moment}

In this section, we prove Theorem \ref{thm:moment}.
Let $k,\ell$ be two even positive integers, and
let $f\in H_k$, $g\in H_\ell$, and $h\in H_{k+\ell}$ be three Hecke eigenforms.
By Watson's formula (see \cite[Theorem 3]{Watson}), we have
\[
  \frac{|\langle fg,h \rangle|^2}{\langle f,f\rangle\langle g,g\rangle\langle h,h\rangle}  =
  \frac{\Lambda(1/2,f\times g\times \bar h)}{4 \Lambda(1,\sym^2 f)\Lambda(1,\sym^2 g)\Lambda(1,\sym^2 \bar h)}
\]
Since $f,g,h$ have real coefficients, we may drop the complex conjugation bars. Here we use the notion of $L$-functions as in \cite[\S3]{Watson}.
The completed symmetric square $L$-function for $f$ is
\[
  \Lambda(s,\sym^2 f) = L_\infty(s,\sym^2 f)  L(s,\sym^2 f) ,
\]
where
\[
  L_\infty(s,\sym^2 f)  = \Gamma_\mathbb{R} (s+1) \Gamma_\mathbb{C}(s+k-1)
\]
with $\Gamma_\mathbb{R}(s)=\pi^{-s/2} \Gamma(s/2)$ and
$\Gamma_\mathbb{C}(s)=\Gamma_\mathbb{R}(s)\Gamma_\mathbb{R}(s+1) = 2(2\pi)^{-s} \Gamma(s)$;
the completed triple production $L$-function for $f, g, h$ is
\[
  \Lambda(s,f\times g\times h) = L_\infty(s,f\times g\times h)  L(s,f\times g\times h) ,
\]
where
\[
  L_\infty(s,f\times g\times h)  = \Gamma_\mathbb{C} (s+1/2) \Gamma_\mathbb{C}(s+k+\ell-3/2) \Gamma_\mathbb{C}(s+k-1/2)\Gamma_\mathbb{C}(s+\ell-1/2).
\]

By some local computations we obtain (see \cite[\S2.3]{BlomerKhanYoung2013distribution})
\begin{equation}\label{eqn:norm2moment}
  \frac{|\langle fg,h \rangle|^2}{\langle f,f\rangle\langle g,g\rangle\langle h,h\rangle}  =
  \frac{\pi^3}{2(k+\ell-1)} \frac{L(1/2,f\times g\times  h)} {L(1,\sym^2 f) L(1,\sym^2 g)L(1,\sym^2 h)}
\end{equation}
By Parseval's identity, we have
\begin{align*}
  \langle fg,fg\rangle & = \sum_{h\in H_{k+\ell}} |\langle fg,h\rangle|^2   \\
  & = \frac{\pi^3}{2(k+\ell-1)} \frac{\langle f,f\rangle\langle g,g\rangle} {L(1,\sym^2 f) L(1,\sym^2 g)}
  \sum_{h\in H_{k+\ell}}\frac{L(1/2,f\times g\times  h)} {L(1,\sym^2 h)}
\end{align*}
Hence we get
\begin{align*}
  \frac{2\pi^2 }{k+\ell-1} \sum_{h\in H_{k+\ell}}\frac{L(1/2,f\times g\times  h)} {L(1,\sym^2 h)}
  = \langle fg,fg\rangle \frac{4}{\pi} \frac{L(1,\sym^2 f) L(1,\sym^2 g)}{\langle f,f\rangle\langle g,g\rangle}.
\end{align*}
Under GRH and GRC, by Theorem \ref{thm:221}, for $\langle f,f\rangle=\langle g,g\rangle=\frac{\pi}{3}$ we have
\begin{align*}
  \frac{2\pi^2 }{k+\ell-1} \sum_{h\in H_{k+\ell}} & \frac{L(1/2,f\times g\times  h)} {L(1,\sym^2 h)} \\
  & = \frac{\pi}{3}(1+O((\log k+\log \ell)^{-1/4+\varepsilon})) \frac{4}{\pi} \frac{L(1,\sym^2 f) L(1,\sym^2 g)}{\pi^2/9} \\
  & = 2 \frac{L(1,\sym^2 f) L(1,\sym^2 g)}{\zeta(2)} (1+O((\log k+\log \ell)^{-1/4+\varepsilon})) .
\end{align*}
Here we have used the fact $\zeta(2)=\pi^2/6$.
This completes the proof of Theorem \ref{thm:moment}.

\appendix

\section{A heuristic argument of Theorems \ref{thm:moment}}

In this appendix, we give a heuristic argument to deduce Theorem \ref{thm:moment}, and hence Theorem \ref{thm:221} via \eqref{eqn:norm2moment}. We will follow \cite{CFKRS} and \cite[\S4]{BlomerKhanYoung2013distribution}.
The following calculations are purely formal. Mimicking the approximation functional equation, we write formally
\[
  L(1/2+\alpha,f\times g\times h)
  = \sum_{m,n} \frac{A_{f\times g}(n,m,1) \lambda_h(n)}{(m^2 n)^{1/2+\alpha}}
  + X_\alpha  \sum_{m,n} \frac{A_{f\times g}(n,m,1) \lambda_h(n)}{(m^2 n)^{1/2-\alpha}}
\]
for certain quantity $X_\alpha$ with $X_0=1$. Here $A_{f\times g}(n,m,\ell)$ is the Fourier coefficients of the Rankin--Selberg convolution $f\times g$. So
\begin{multline*}
  \frac{2\pi^2}{k+\ell-1} \sum_{h\in H_{k+\ell}} \frac{
    L(1/2,f\times g\times  h )}{L(1,\sym^2 h)}
    \\
    = \frac{2\pi^2}{k+\ell-1} \sum_{h\in H_{k+\ell}} \frac{
    1}{L(1,\sym^2 h)} \sum_{m,n} \frac{A_{f\times g}(n,m,1) \lambda_h(n)}{(m^2 n)^{1/2+\alpha}}  + \cdots,
\end{multline*}
where the dots indicate another similar term.
By the Petersson formula, this spectral sum can be expressed as a diagonal term plus a sum of Kloosterman sums.
The recipe of \cite{CFKRS} instructs us to retain only the diagonal terms. Thus we obtain
\[
  \frac{2\pi^2}{k+\ell-1} \sum_{h\in H_{k+\ell}} \frac{
    L(1/2,f\times g\times  h )}{L(1,\sym^2 h)}
    \sim  \sum_{m} \frac{A_{f\times g}(1,m,1)}{m^{1+2\alpha}}  + \sum_{m} \frac{A_{f\times g}(1,m,1)}{m^{1-2\alpha}} .
\]
By \cite[Theorem 1]{Kontorovich}, the exterior square $L$-function of $f\times g$ is given by
\[
  L(s,f\times g,\wedge^2) = \zeta(2s) \sum_{m\geq1} \frac{A_{f\times g}(1,m,1)}{m^s}.
\]
At this point we can set $\alpha=0$. So we get
\[
  \frac{2\pi^2}{k+\ell-1} \sum_{h\in H_{k+\ell}} \frac{
    L(1/2,f\times g\times  h )}{L(1,\sym^2 h)}
    \sim  2 \frac{L(1,f\times g,\wedge^2)}{\zeta(2)}.
\]
Let $\{\alpha_i(p)\}_{i=1}^4=\{\alpha_f(p)\alpha_g(p),\alpha_f(p)\alpha_g(p)^{-1},\alpha_f(p)^{-1}\alpha_g(p),\alpha_f(p)^{-1}\alpha_g(p)^{-1}\}$ be the Satake parameters of $f\times g$ at prime $p$. The Euler product of $L(s,f\times g,\wedge^2)$ is given by
\[
  L(s,f\times g,\wedge^2)
  = \prod_{p} \prod_{1\leq i<j\leq 4} \left( 1-\frac{\alpha_i(p)\alpha_j(p)}{p^s} \right)^{-1}.
\]
A simple computation shows that
\begin{align*}
  L(s,f\times g,\wedge^2)
  & = \prod_{p} \prod_{-1\leq i\leq 1} \left( 1-\frac{\alpha_f(p)^{2i}}{p^s} \right)^{-1}
  \prod_{-1\leq j\leq 1} \left( 1-\frac{\alpha_g(p)^{2j}}{p^s} \right)^{-1} \\
  & = L(s,\sym^2 f) L(s,\sym^2 g).
\end{align*}
Hence we obtain
\[
  \frac{2\pi^2}{k+\ell-1} \sum_{h\in H_{k+\ell}} \frac{
    L(1/2,f\times g\times  h )}{L(1,\sym^2 h)}
    \sim  2 \frac{L(1,\sym^2 f) L(1,\sym^2 g)}{\zeta(2)}.
\]

\section*{Acknowledgements}
The author would like to thank Shenghao Hua and Liangxun Li for letting him to write this paper at the beginning of our project, and for many discussions on \cite{HHL}. 
He wants to thank Steve Lester and Ze\'ev Rudnick  for helpful comments.

This material is based upon work supported
by the Swedish Research Council under grant no. 2021-06594
while the author was in residence at Institut Mittag-Leffler in Djursholm, Sweden
during the spring semester of 2024.



\begin{thebibliography}{10}

\bibitem{BB}
V. Blomer and F. Brumley, 
Simultaneous equidistribution of toric periods and fractional moments of $L$-functions. 
\emph{J. Eur. Math. Soc. (JEMS)} 26 (2024), no. 8, 2745--2796.

\bibitem{BBK}
V. Blomer, F. Brumley, and I. Khayutin, 
The mixing conjecture under GRH.
To appear in \emph{Ann. Sci. \'Ec. Norm. Sup\'er. (4)}, 2025+. 
\href{https://arxiv.org/abs/2212.06280}{arXiv:2212.06280}

\bibitem{BlomerKhanYoung2013distribution}
V.~Blomer, R.~Khan, and M.~Young,
\newblock Distribution of mass of holomorphic cusp forms.
\newblock {\em Duke Math. J.} 162 (2013), no. 14, 2609--2644.


\bibitem{ButtcaneKhan}
J. Buttcane and R. Khan,
On the fourth moment of Hecke--Maass forms and the random wave conjecture.
\emph{Compos. Math.} 153 (2017), no. 7, 1479--1511.



\bibitem{Chandee}
V. Chandee,
Explicit upper bounds for $L$-functions on the critical line.
\emph{Proc. Amer. Math. Soc.} 137 (2009), no. 12, 4049--4063.


\bibitem{CFKRS}
J. B. Conrey, D. W. Farmer, J. P. Keating, M. O. Rubinstein, and N. C. Snaith,
Integral moments of $L$-functions.
\emph{Proc. Lond. Math. Soc. (3)} 91 (2005), 33--104.

\bibitem{CI}
J. B. Conrey and H. Iwaniec,
The cubic moment of central values of automorphic $L$-functions.
\emph{Ann. of Math. (2)} 151 (2000), no. 3, 1175--1216.


\bibitem{Constantinescu}
P. Constantinescu,
Dissipation of correlations of holomorphic cusp forms.
\emph{ArXiv preprint} (2021), 22 pp.
\href{https://arxiv.org/abs/2112.01427}{arXiv:2112.01427}




\bibitem{Deligne}
P. Deligne,
La conjecture de Weil. I.
\emph{Inst. Hautes \'Etudes Sci. Publ. Math.} No. 43 (1974), 273--307.







\bibitem{Holowinsky}
R. Holowinsky, 
Sieving for mass equidistribution. 
\emph{Ann. of Math. (2)} 172 (2010), no. 2, 1499--1516.

\bibitem{HS}
R. Holowinsky and K. Soundararajan,
Mass equidistribution for Hecke eigenforms.
\emph{Ann. of Math. (2)} 172 (2010), no. 2, 1517--1528.



\bibitem{HHL}
S. Hua, B. Huang, and L. Li,
Joint value distribution of Hecke--Maass forms.
\emph{ArXiv Preprint} (2024), 22 pp.
\href{https://arxiv.org/abs/2405.00996}{arXiv:2405.00996}


\bibitem{Huang}
B. Huang,
Effective decorrelation of Hecke eigenforms.
\emph{Trans. Amer. Math. Soc.} 377 (2024), no. 9, 6669--6693.


\bibitem{HuangLester}
B. Huang and S. Lester,
Quantum variance for dihedral Maass forms.
\emph{Trans. Amer. Math. Soc.} 376 (2023), no. 1, 643--695.

\bibitem{HLWY} B. Huang, S. Lester, I. Wigman, and N. Yesha,
On the supremum of random cusp forms. 
\emph{ArXiv Preprint} (2025), 23 pp.
\href{https://arxiv.org/abs/2508.16813}{arXiv:2508.16813}




\bibitem{IwaniecKowalski2004analytic}
H.~Iwaniec and E.~Kowalski.
\newblock {\em Analytic number theory}, volume~53 of {\em American Mathematical
  Society Colloquium Publications}.
\newblock American Mathematical Society, Providence, RI, 2004.

\bibitem{KatokSarnak}
S. Katok and P. Sarnak,
Heegner points, cycles and Maass forms.
\emph{Israel J. Math.} 84 (1993), no. 1-2, 193--227.





\bibitem{KohnenZagier}
W. Kohnen and D. Zagier, Values of $L$-series of modular forms at the center of the critical strip.
\emph{Invent. Math.} 64 (1981), 175--198.

\bibitem{Kontorovich}
A. Kontorovich,
The Dirichlet series for the exterior square $L$-function on $\GL(n)$.
\emph{Ramanujan J.} 21 (2010), no. 3, 263--266.


\bibitem{KurlbergRudnick}
P. Kurlberg and Z. Rudnick,
Value distribution for eigenfunctions of desymmetrized quantum maps.
\emph{Internat. Math. Res. Notices} 2001, no. 18, 985--1002.



\bibitem{Lapid}
E. M. Lapid, On the nonnegativity of Rankin-Selberg $L$-functions at the center of symmetry.
\emph{Int. Math. Res. Not. IMRN} 2 (2003), 65--75.



\bibitem{LL}
Y.-K. Lau and G. L\"u,
Sums of Fourier coefficients of cusp forms. 
\emph{Q. J. Math.} 62 (2011), no. 3, 687--716.

\bibitem{LW}
Y.-K. Lau and J. Wu,
A density theorem on automorphic $L$-functions and some applications.
\emph{Trans. Amer. Math. Soc.} 358 (2006), no. 1, 441--472.

\bibitem{LR}
S. Lester and M. Radziwi\l\l, 
Quantum unique ergodicity for half-integral weight automorphic forms. 
\emph{Duke Math. J.} 169 (2020), no. 2, 279--351.








%







\bibitem{RS}
Z. Rudnick and P. Sarnak,
The behaviour of eigenstates of arithmetic hyperbolic manifolds.
\emph{Comm. Math. Phys.} 161 (1994), no. 1, 195--213.

\bibitem{RudnickSoundararajan}
Z. Rudnick and K. Soundararajan,
Lower bounds for moments of $L$-functions: symplectic and orthogonal examples.
\emph{Multiple Dirichlet series, automorphic forms, and analytic number theory}, 293--303,
Proc. Sympos. Pure Math., 75,
\emph{Amer. Math. Soc., Providence, RI}, 2006.


\bibitem{soundararajan2009moments}
K. Soundararajan,
Moments of the Riemann zeta function.
\emph{Ann. of Math. (2)} 170 (2009), no. 2, 981--993.

\bibitem{Soundararajan}
K. Soundararajan,  
Weak subconvexity for central values of $L$-functions. 
\emph{Ann. of Math. (2)} 172 (2010), no. 2, 1469--1498.

\bibitem{SY}
K. Soundararajan and M. Young,
The second moment of quadratic twists of modular $L$-functions.
\emph{J. Eur. Math. Soc. (JEMS)} 12 (2010), no. 5, 1097--1116.



\bibitem{Watson}
T. Watson,
Rankin Triple Products and Quantum Chaos.
\emph{ArXiv preprint} (2008), 66 pp.
\href{https://arxiv.org/abs/0810.0425v3}{arXiv:0810.0425v3}

\bibitem{Xia}
H. Xia, On $L^\infty$ norms of holomorphic cusp forms.
\emph{J. Number Theory} 124 (2007), no. 2, 325--327.

\bibitem{Young}
M. Young,
Weyl-type hybrid subconvexity bounds for twisted $L$-functions and Heegner points on shrinking sets.
\emph{J. Eur. Math. Soc. (JEMS)} 19 (2017), no. 5, 1545--1576.


\bibitem{Zenz}
P. Zenz,
Sharp bound for the fourth moment of holomorphic Hecke cusp forms.
\emph{Int. Math. Res. Not. IMRN} 2023, no. 16, 13562--13600.

\end{thebibliography}
\end{document}